\documentclass{amsart}

\usepackage{amssymb,color}
\usepackage{amsfonts}
\usepackage{amsmath}
\usepackage{euscript}
\usepackage{enumerate}
\usepackage{graphics}

\newtheorem{theorem}{Theorem}[section]
\newtheorem{lemma}[theorem]{Lemma}
\newtheorem{note}[theorem]{Note}
\newtheorem{prop}[theorem]{Proposition}

\newtheorem*{Theorem1'}{Theorem 1'}

\theoremstyle{definition}

\theoremstyle{remark}

\numberwithin{equation}{section}

\newcommand \g{{\mathfrak g}}

\renewcommand \a{{\mathfrak a}}
\newcommand \w{{\mathfrak b}}
\newcommand \e{{\mathfrak e}}

\newcommand \ad{{\mathrm {ad}}}

\newcommand \gl{{\mathfrak {gl}}}
\renewcommand \sl{{\mathfrak {sl}}}

\newcommand \B{{\mathcal B}}
\newcommand \GL{{\mathrm {GL}}}

\newcommand \T{{\mathcal{S}}}

\newcommand \al{{\alpha}}

\newcommand \Q{{\mathbb Q}}
\newcommand \chr{{\mathrm {char}}}

\begin{document}

\title[Indecomposable modules of solvable Lie algebras ]
{Indecomposable modules of 2-step solvable Lie algebras in arbitrary characteristic}

\author{Leandro Cagliero}
\address{CIEM-CONICET, FAMAF-Universidad Nacional de C\'ordoba, C\'ordoba, Argentina.}
\email{cagliero@famaf.unc.edu.ar}
\thanks{The first author was supported in part by CONICET and SECYT-UNC grants.}

\author{Fernando Szechtman}
\address{Department of Mathematics and Statistics, University of Regina, Canada}
\email{fernando.szechtman@gmail.com}
\thanks{The second author was supported in part by an NSERC discovery grant}

%    General info

%\date{January 1, 2001 and, in revised form, June 22, 2001.}

%\dedicatory{This paper is dedicated to our advisors.}


\begin{abstract} Let $F$ be an algebraically closed field and
consider the Lie algebra $\g=\langle x\rangle\ltimes \a$, where
$ad\, x$ acts diagonalizably on the abelian Lie algebra $\a$.
Refer to a $\g$-module as admissible if $[\g,\g]$ acts via
nilpotent operators on it, which is automatic if $\chr(F)=0$. In
this paper we classify all indecomposable $\g$-modules $U$ which are admissible as well as uniserial,
in the sense that $U$ has a unique composition series.

%When $\chr(F)=0$ we recover results previously obtained elsewhere,
%although the present methods are drastically simpler. When $F$ has
%prime characteristic $p$ the classification of admissible
%uniserial $\g$-modules of length $m\leq p$ is essentially the same
%as in characteristic 0, although when~$m>p$ the modular analogues
%of the isomorphism of classes from characteristic~0 now split into
%several distinct classes, according to the orbits of an
%intransitive group action, while, on the other hand, infinitely many
%new isomorphism classes arise which only exist in prime characteristic.
\end{abstract}

\maketitle

\section{Introduction}

We fix throughout an arbitrary field $F$. All vector spaces,
including all Lie algebras and their modules, are assumed to be finite dimensional over $F$.

If we consider the problem of classifying all indecomposable modules of a Lie algebra $\g$,
there are two well-known and opposite instances where a solution is available.

In the first case $\g=\langle x\rangle$ is 1-dimensional.
Then a $\g$-module $U$ is indecomposable if and only if $x$ acts on $U$ via the
companion matrix of the power of an irreducible polynomial.
In particular, when $F$ is algebraically closed, this
means that $x$ acts on $U$ through a Jordan block.

In the second case $F$ is algebraically closed of characteristic 0
and $\g$ is semisimple. Then the indecomposable $\g$-modules
are the irreducible $\g$-modules, which can be classified by their highest weights
(see \cite{H}).

In general, the problem seems to be untractable, as discussed in \cite{GP}
for abelian Lie algebras of dimension $\geq 2$, in \cite{S} for the 3-dimensional
complex euclidean Lie algebra $\e(2)$, and in \cite{M} for virtually any complex
Lie algebra other than semisimple or 1-dimensional.

In spite of the difficulties in achieving a classification,
various families of indecomposable modules over various types of non-semisimple
complex Lie algebras have been constructed and classified in recent years. See, for instance,
\cite{CMS}, \cite{CM}, \cite{DG}, \cite{DP}, \cite{DKR} and \cite{J}.

One may hope to contribute to the classification problem by restricting attention to certain types
of indecomposable modules, defined by suitable structural properties. In a sense, the simplest type of indecomposable module,
other than irreducible, is a uniserial module, i.e., one possessing a unique composition series.

In trying to classify all uniserial modules of an abelian Lie
algebra $\g$, the first observation is that when $\dim(\g)=1$ then
uniserial means indecomposable, a case considered above. If
$\dim(\g)\geq 2$ we immediately run into technical problems, due
to the existence of field extensions without a primitive element
(see \cite{Mc}, Ch. 1). However, it is shown in \cite{CS2} that if
$F$ is a sufficiently large perfect field and $U$ is a uniserial
$\g$-module then there is $x\in\g$ such that $U$ is a uniserial
$F[x]$-module. Thus, relative to a suitable basis, the action of
$x$ on $U$ can be represented by the companion matrix of the power
of an irreducible polynomial and every other element of $\g$ can
be represented as a polynomial in this matrix. The proof of this
result is quite technical, relying on subtle results from Galois
theory and the existence and uniqueness of the Jordan-Chevalley
decomposition over perfect fields. We begin this paper by giving
an elementary proof, under the stronger assumption that $F$ is
algebraically closed, in which case we obtain the sharper
conclusion that, if $\B$ is any basis of $\g$, then $x$ can be
taken from $\B$, which is certainly false for more general perfect
fields (look at $\Q[\sqrt{2},\sqrt{3}]$ as a module over the
2-dimensional Lie algebra $\g=\langle \sqrt{2},\sqrt{3}\rangle$).
%See Theorem \ref{comm} for details.

Let us consider next uniserial modules over a non-abelian Lie
algebra $\g$. For $\g=\sl(2)\ltimes V(m)$, where $V(m)$ is the
irreducible $\sl(2)$-module of highest weight $m\geq 1$, a
complete solution was achieved in \cite{CS}. The classification of
all uniserial $\sl(2)\ltimes V(m)$-modules is intimately related
to the problem of determining all non-trivial zeros of the
Racah-Wigner $6j$-symbol within certain parameters, which is, in
general, a fairly difficult problem (see \cite{R} and \cite{L}).
We remark that in the special case $\g=\sl(2)\ltimes V(1)$, the
classification of uniserial and far more general indecomposable
modules had been previously obtained in \cite{P}.

Suppose next $\g=\langle x\rangle\ltimes\a$, where $ad\, x$ acts
diagonalizably on the abelian Lie algebra~$\a$. In this case, a
full classification of uniserial and far more general
indecomposable $\g$-modules was obtained in \cite{CS3} under the
only assumption that $F$ have characteristic~0. This
classification fails completely in prime characteristic, the main
reason being the failure of Lie's theorem and, as a result, the
failure of $[\g,\g]$ to act trivially on an irreducible
$\g$-module. Let us refer to a $\g$-module $U$ as admissible if
$[\g,\g]$ acts on $U$ via nilpotent operators, which is automatic
if $\chr(F)=0$. In this paper we classify all admissible uniserial
$\g$-modules of composition length $m\geq 1$ when $F$ is
algebraically closed, whatever the characteristic of $F$. If
$\chr(F)=0$ or $\chr(F)=p\geq m$ our results match those from
\cite{CS3}, although, taking advantage of the fact that $F$ is
algebraically closed, are proofs are drastically simpler. When $F$
has prime characteristic $p<m$ each of the isomorphism of classes
that arise when $p\geq m$ now split into infinitely distinct
classes, according to the orbits of an intransitive group action,
while, on the other hand, infinitely many new isomorphism classes
arise which only exist when $p<m$.

\section{Uniserial families of commuting operators}
\label{s2}

%We assume throughout this section that $F$ is algebraically closed.

It will be convenient to use a slightly simpler language than above.
A family $\T$ of endomorphisms of a non-zero vector space $V$ is
said to be uniserial if the $\T$-invariant subspaces of
$V$ form a chain under inclusion.

\begin{lemma}\label{superd}
Let $\T$ be a uniserial family of operators acting on a non-zero
vector space $V$. Suppose that there is a basis of $V$ relative to
which every element of $\T$ can be represented by an upper
triangular matrix. Then none of the entries along the first
superdiagonal of these matrices is identically 0.
\end{lemma}

\begin{proof} Given any non-zero (resp. proper) $\T$-invariant subspace $W$ of $V$ we see
that $\T$ induces a uniserial family of linear operators on $W$
(resp. $V/W$). Thus, if the result is false, we obtain a uniserial
family of $2\times 2$ diagonal matrices, which is absurd.
\end{proof}

\begin{lemma}\label{upper} Suppose $A\in M_m(F)$ is upper triangular.
Then there is $P\in\GL_m(F)$ such that $P$ is upper triangular with diagonal entries equal to 1
and $B=P^{-1}AP$ satisfies: if $B_{ii}\neq B_{jj}$ then
$B_{ij}=0$.
\end{lemma}

\begin{proof} We will clear all required blocks of $A$ by means of transformations
$$T\mapsto
(I-\alpha E^{ij})T(I+\alpha E^{ij}),
$$
where $T\in M_m(F)$, $i<j$ and $\alpha\in F$, proceeding upwards one row at a time and, within each row, from
left to right one column at a time.

Suppose that $C$ was obtained from $A$ by means of a sequence of these
transformations and satisfies the following: for some $1\leq i<n$,
all required entries of $C$ are 0 below row~$i$, and there is
$i<j\leq n$ such that if $i<k<j$ and $C_{ii}\neq C_{kk}$
then $C_{ik}=0$. Set
$$
D=(I-\al E^{ij})C(I+\al E^{ij}),
$$
where $\al=0$ if $C_{ii}=C_{jj}$ and $\al=C_{ij}/(C_{jj}-C_{ii})$ otherwise.
Then the entries of $D$ and $C$ coincide below row
$i$ as well as within row $i$ but to the left of position
$(i,j)$. Moreover, $D_{ij}=0$ if
$D_{ii}\neq D_{jj}$. We may thus continue this process
and find $B$ as required.
\end{proof}

\begin{theorem}\label{comm} Suppose $F$ is algebraically closed field and let $\T=\{x_1,\dots,x_n\}$ be a uniserial
family of commuting endomorphisms of a non-zero vector space $V$. Then there is an index $1\leq i\leq n$
and a basis of $V$ relative to which $x_i$ is represented by an upper triangular Jordan block $J_m(\alpha)$, $\alpha\in F$, and all other $x_j$
are represented by a polynomial in $J_m(\alpha)$.
\end{theorem}

\begin{proof} Since $F$ is algebraically closed and $x_1,\dots,x_n$ commute
there is a basis of $V$ relative to which every $x_i$ is represented by an upper triangular matrix, say $A_i$.

Given $1\leq j\leq n$, we claim that all eigenvalues
of $A_j$ are equal. If not, $A_j$ has a principal $2\times 2$
submatrix, corresponding to entries $i,i+1$, equal to
$$
\left(%
\begin{array}{cc}
  a & c \\
  0 & b \\
\end{array}%
\right),
$$
where $a\neq b$. By Lemma \ref{upper}, we may assume that $c=0$.
Since $A_j$ commutes with all $A_k$, it follows that entry $(i,i+1)$ of every
$A_k$ is 0, which contradicts Lemma~\ref{superd}. This
proves the claim.

By Lemma \ref{superd}, there is an index $i$ satisfying $1\leq i\leq n$ such that entry $(1,2)$ of  $A_i$ is not 0.
Since $A_1,\dots,A_n$ commute, Lemma \ref{superd} implies that
all other entries of the first superdiagonal of $A_i$ are non-zero.
Conjugating by a suitable matrix, we may assume that $A_i=J_m(\al)$,
where $\al$ is the only eigenvalue of $A_i$. Since the only matrices commuting with $J_m(\al)$ are polynomials in
it, the result follows.
\end{proof}

\section{Uniserial modules over 2-step solvable Lie algebras}
\label{s3}

We assume henceforth that $F$ is algebraically closed and that
$\g=\langle x\rangle\ltimes\a$, where $\ad\,x$ acts diagonalizably
on the abelian Lie algebra $\a$. If $U$ is a uniserial $\g$-module
acted upon trivially by $[\g,\g]$ then $U$ is a uniserial module
over the abelian Lie algebra $\g/[\g,\g]$, a case described in
Theorem \ref{comm}. We may thus restrict attention to the case
when $U$ is admissible but not annihilated by $[\g,\g]$. We
reiterate the observation made in the Introduction that, thanks to
Lie's theorem, the condition that $U$ be admissible is utterly
redundant when $\chr(F)=0$.
%In this case, Lie's theorem implies that $U$ cannot be irreducible, i.e., the composition length of $U$
%must be $>1$.

Throughout this section we let $J_m(0)$ stand for the upper triangular $m\times m$ Jordan block
with eigenvalue 0. Moreover, given $\delta\in F$, we let $\a(\delta)=\{v\in \a\,|\, [x,v]=\delta v\}$.

\begin{prop}\label{repe} Let $p$ be a prime and suppose $F$ has
characteristic 0 or $p$. Given $m\geq 1$, let $G$
be the unipotent part of unit group of $F[J]$, $J=J_m(0)$, namely the abelian subgroup of $\GL_m(F)$ formed by all $I+f(J)$ such that
$f\in F[X]$ has no constant term. Further, given $\al\in F$, set $D=\mathrm{diag}(\al,\al-1,\dots,\al-(m-1))$
and let ${\mathcal Y}$ be the subset of $M_m(F)$ of all $D+f(J)$, where $f\in F[X]$ has no constant term. Then

{\rm (1)}  $G$ acts on ${\mathcal Y}$ by conjugation.

{\rm (2)} If $\chr(F)=0$ then the stabilizer $G_Y$ of any $Y\in {\mathcal Y}$
is trivial.

{\rm (3)} If $\chr(F)=p$ and $Y\in {\mathcal Y}$ then $G_Y$
consists of all $I+h(J^p)$ such that $h\in F[X]$ has no constant
term. In particular, if $p\geq m$ then $G_Y$ is trivial.

{\rm (4)} If $\chr(F)=0$ or $\chr(F)=p\geq m$ then the action of $G$ on
${\mathcal Y}$ is regular.

%{\rm (5)} If $Y\in {\mathcal Y}$ then the $G$-orbit of $Y$
%consists of all
%$$
%Y-(I+g(J)+\cdots+g(J)^{m-1})g'(J)J,
%$$
%where $g\in F[X]$ has no constant term and $g'$ is the formal
%derivative of $g$.

{\rm (5)} If $\chr(F)=p$ then $D+F[J^p]J$ a set of representatives for the action of $G$ on ${\mathcal Y}$.
\end{prop}

\begin{proof} Note first of all that
\begin{equation}
\label{relbas} [D,J]=J.
\end{equation}
Given $T\in G$, set $E=TDT^{-1}$. Clearly $TJT^{-1}=J$, so
(\ref{relbas}) yields $[E,J]=J$. Thus $E-D\in F[J]$, the
centralizer of $J$ in $M_m(F)$. But $E$ and $D$ have the same
diagonal entries, so $E=D+f(J)$, where $f\in F[X]$ has no constant
term. This shows that ${\mathcal Y}$ is stable under conjugation
by $G$ and gives (1).

It is readily seen that all $Y\in {\mathcal Y}$ have the same
stabilizer in $G$. Now by (\ref{relbas})
$$
[D,J^k]=k J^k,\quad k\geq 0,
$$
whence
\begin{equation}
\label{relpol} [D,g(J)]=g'(J)J,\quad g\in F[X],
\end{equation}
where $g'$ is the formal derivative of $g$. We apply (\ref{relpol}) to the case $g=1+f$, where $f\in F[X]$ has
no constant term and, moreover, $f=0$ or $\deg(f)<m$. Then
$[D,I+f(J)]=0$ if and only if $X^{m-1}|f'$, i.e.,
$f'=0$. If $\chr(F)=0$ this means $f=0$, while if $\chr(F)=p$ it
means $f(X)=h(X^p)$, where $h\in F[X]$ has no constant term. This
proves (2) and~(3).

Suppose next that $\chr(F)=0$ or $\chr(F)=p\geq m$. To see that $G$ acts transitively
on ${\mathcal Y}$, let
$Y=D+f(J)\in {\mathcal Y}$, where $f\in F[X]$ has no constant
term. Then $[Y,J]=J$. Moreover, $Y$ and $D$ have the same diagonal
entries, which are {\em distinct}. We infer from Lemma \ref{upper}
the existence of an upper triangular matrix $P\in\GL_m(F)$ with
diagonal entries equal to 1 such that $P^{-1}YP=D$. Set
$L=P^{-1}JP$. It follows from (\ref{relbas}) that $[D,L]=L$. Our assumption on $F$ implies that
the only matrices in $\gl(m)$ that are eigenvectors for $\ad\, D$ with eigenvalue 1 are the conjugates of
$J$ by a diagonal matrix. But $P$ has diagonal entries equal to 1,
so the entries along the first superdiagonal of $L$ are equal to
1. This shows that $L=J$. Therefore $[P,J]=0$, which implies $P\in
G$. This proves (4).

Given $g\in F[X]$ with no constant term, we can use (\ref{relpol}) to calculate
\begin{equation}
\label{conj}
\begin{aligned}
(I-g(J))^{-1}D(I-g(J))&=(I+g(J)+\cdots+g(J)^{m-1})D(I-g(J))\\
&=D-(I+g(J)+\cdots+g(J)^{m-1})g'(J)J.
\end{aligned}
\end{equation}
Now, given $T\in G$, we have
$$
T=(I-a_1J)(I-a_2J^2)\cdots(I-a_{m-1}J^{m-1})
$$
for unique $a_1,a_2,\dots,a_{m-1}\in F$. Thus, every $G$-conjugate of $D$ can be obtained by successively conjugating $D$ by
$I-a_1 J,I-a_2J^2,\dots,I-a_{m-1}J^{m-1}$. This can be computed by means of (\ref{conj}), to obtain
$$
T^{-1}DT=D-(b_1J+b_2J^2+\cdots+b_{m-1}J^{m-1}),
$$
\begin{equation}
\label{beca}
b_k=\underset{d|k}\sum d a_d^{k/d},\quad 1\leq k<m.
\end{equation}
Suppose $\chr(F)=p$. Thus, if $1\leq k<m$, $d|k$ and $p|d$, then $a_d$ is not present in~$b_k$. It follows from (\ref{beca})
that if $b_i=0$ for all $i\not\equiv 0\mod p$ then all $b_1,\dots,b_{m-1}$ are equal to 0. This shows that the elements of $D+F[J^p]J$ are in distinct $G$-orbits. On the other hand,
given any $c_1,\dots,c_{m-1}\in F$ we can recursively find $a_i$, with $1\leq i<m$ relatively prime to $p$,
so that $b_i=c_i$ for all such $i$. This shows that every element of ${\mathcal Y}$ is $G$-conjugate to one in $D+F[J^p]J$,
which completes the proof of (5).
\end{proof}

%\begin{note}
%If $\chr(F)=0$ or $\chr(F)=p\geq m$ then given any $c_1,\dots,c_{m-1}\in F$ we can recursively find $a_1,\dots,a_{m-1}$
%so that $b_1=c_1,\dots,b_{m-1}=c_{m-1}$, which gives another proof of the transitivity of the action of $G$ in this case.
%\end{note}

\begin{prop}\label{auxi} Let $U$ be an admissible uniserial $\g$-module of composition length~$m$ not annihilated by $[\g,\g]$.
Then $m=\dim(U)>1$. Moreover, after replacing $x$ by a suitable
scalar multiple, there is an eigenvector $v\in\a$ of $\ad\,x$ of
eigenvalue 1 and a basis of $U$ relative to which $x$ is
represented by an upper triangular matrix $Y\in\gl(m)$ with
diagonal entries $\al,\al-1,\cdots,\al-(m-1)$ for some $\al\in F$
and $v$ is represented by $J=J_m(0)$.

Moreover, let $D=\mathrm{diag}(\al,\al-1,\dots,\al-(m-1))$.
If $\chr(F)=0$ we can take $Y=D$, while
if $F$ has prime characteristic $p$ we can take $Y\in D+F[J^p]J$.
\end{prop}

\begin{proof} Let $0=U_1\subset U_1\subset\cdots\subset U_m$ be a composition series of $U$. Given any $1\leq i\leq m$ set
$V_i=U_i/U_{i-1}$ and $W_i=\{v\in V_i\, |\, [\g,\g]v=0\}$. As
$[\g,\g]$ is an ideal of $\g$, it follows that $W_i$ is a
$\g$-invariant subspace of $V_i$. Since $U$ is admissible
$[\g,\g]$ acts via nilpotent operators on $V_i$, so $W_i\neq 0$.
But $V_i$ is irreducible, so $V_i=W_i$. Thus $V_i$ is an
irreducible module over the abelian Lie algebra $\g/[\g,\g]$.
Since $F$ is algebraically closed, we deduce that $\dim(V_i)=1$.
As this happens to every $i$, we infer $\dim(U)=m$. But $U$ itself
is not annihilated by $[\g,\g]$, so $m>1$. Let
$\B=\B_1\cup\cdots\cup\B_m$ be a basis of $U$ such that
$\B_1\cup\cdots\B_i$ is a basis of $U_i$ for each $1\leq i\leq m$.
Then, relative to $\B$, every $y\in\g$ is represented by an upper
triangular matrix $M(y)\in \gl(m)$, so $M(y)$ is strictly upper
triangular for every $y\in [\g,\g]$. Let $A=M(x)$. By Lemma
\ref{upper} we may assume that $A_{ij}=0\text{ whenever
}A_{ii}\neq A_{jj}$ and we will make this assumption.

We claim that no two consecutive diagonal entries of $A$ are
equal. Suppose, if possible, that this is false.

\smallskip

\noindent{\sc Case 1.} All diagonal entries of $A$ are the same.

\smallskip

Since $\ad\, x$ acts diagonalizably on~$\a$, it is clear that $[\g,\g]$ is the sum of all
$\a(\delta)$ with $\delta\neq 0$. As $[\g,\g]$ does
not act trivially on $U$, there is a non-zero $\delta\in F$ and a non-zero $y\in \a(\delta)$ such that $y U\neq 0$.
Thus $[x,y]U\neq 0$ and a fortiori $M(y)\neq 0$. Therefore $\delta$ is an eigenvalue when $\ad\, A$ acts $\gl(m)$.
On the other hand, all diagonal entries of $A$ are equal, so $\ad\, A$ acts nilpotently on $\gl(m)$, whence
all eigenvalues when $\ad\, A$ acts on $\gl(m)$ are equal to~0, a contradiction.

\smallskip

\noindent{\sc Case 2.} $A$ has a triple of diagonal entries
$(a,a,b)$, say in positions $i,i+1,i+2$, where $a\neq b$.

\smallskip

Given $y\in\g$ let $N(y)$ (resp. $P(y)$) be the $2\times 2$ (resp. $3\times 3$) submatrix of $M(y)$
corresponding to rows and columns $i,i+1$ (resp. $i,i+1,i+2$).

Clearly, entry $(i,i+1)$ of $M(y)$ is 0 for all $y\in\a(\delta)$,
$\delta\neq 0$. Thus $U_{i+1}/U_{i-1}$ is a uniserial module over
the abelian Lie algebra $\langle x\rangle\oplus\a(0)$, so by
Theorem \ref{comm} the diagonal entries of $N(y)$ are equal for
every $y\in \langle x\rangle\oplus\a(0)$.

Since $A_{i+1,i+2}=0$, Lemma \ref{superd} implies that entry $(i+1,i+2)$ of $M(y)$ must be non-zero for some $y\in\a(\delta)$.
As $a\neq b$, we see that necessarily $\delta\neq 0$.
Thus
$$
P(y)=\left(\begin{array}{ccc}
  0 & 0 & d \\
  0 & 0 & e \\
  0 & 0 & 0 \\
\end{array}%
\right),
$$
where $e\neq 0$. By Lemma \ref{superd}, either $A_{i,i+1}\neq 0$
or else entry $(i,i+1)$ of some $M(z)$, $z\in\a(0)$, must be
non-zero.

Suppose first that $A_{i,i+1}\neq 0$. Then
$$
P(x)=\left(%
\begin{array}{ccc}
  a & c & 0 \\
  0 & a & 0 \\
  0 & 0 & b \\
\end{array}%
\right),
$$
where $c\neq 0$. It follows that entries $(i,i+2)$
and $(i+1,i+2)$ of $[A,M(y)]$ are respectively equal to $(a-b)d+ce$
and $(a-b)e$. But $[A,M(y)]=\delta M(y)$, so $a-b=\delta$, which
implies $ce=0$, a contradiction.

Suppose next $A_{i,i+1}=0$ but entry $(i,i+1)$ of $M(z)$ is non-zero form some $z\in\a(0)$. From $[A,M(z)]=0$ we infer
that entries $(i,i+2)$ and $(i+1,i+2)$ of $M(z)$ are equal to 0. Thus
$$
P(z)=\left(%
\begin{array}{ccc}
  f & k & 0 \\
  0 & f & 0 \\
  0 & 0 & g \\
\end{array}%
\right),
$$
where $k\neq 0$. From $[M(y),M(z)]=0$ we deduce $g=f$ and a fortiori
$ke=0$, a contradiction.

\smallskip

\noindent{\sc Case 3.} $A$ has a triple of diagonal entries
$(b,a,a)$, say in positions $i,i+1,i+2$, where $a\neq b$.

\smallskip

This case can be handled as above, interchanging the roles
of $i,i+1$ by those of $i+1,i+2$.

\smallskip

This proves the claim. Since all consecutive diagonal entries of
$A$ are distinct, all entries along the first superdiagonal of $A$
are equal to 0. By Lemma \ref{superd} entry $(1,2)$ of some
$M(v)$, $v\in\a(\delta)$, must be non-zero. Clearly $\delta\neq
0$. Since all $M(z)$, $z\in\a$, commute, Lemma \ref{superd}
implies that all entries along the first superdiagonal of the same
$M(v)$ are non-zero. Replace $x$ by $\delta^{-1}x$. Then $[x,v]=v$
forces $A$ to have diagonal entries $\al,\al-1,\dots,\al-(m-1)$
for some $\al\in F$.

Conjugating all $M(z)$, $z\in \g$, by a suitable upper triangular
matrix results in a new matrix representation $z\mapsto Q(z)$ such
$Q(v)=J$ and $Q(x)=Y$ has the same diagonal entries
$\al,\al-1,\dots,\al-(m-1)$ as $D$. Since $[Y-D,J]=0$ and the
centralizer of $J$ in $\gl(m)$ is $F[J]$, it follows that
$Y=D+f(J)$, where $f\in F[X]$ has no constant term. Now apply
Proposition \ref{repe}.
\end{proof}

\begin{prop}\label{exist} Suppose that $\chr(F)=0$ or $\chr(F)=p\geq m$.
Assume that $1$ is an eigenvalue when $\ad\, x$ acts
on $\a$ and let $v\in\a$ be a corresponding eiegenvalue. Given
$m>1$ and $\al\in F$, set
$D=\mathrm{diag}(\al,\al-1,\dots,\al-(m-1))$ and $J=J_m(0)$, and let
$f^i:\a(i)\to F$, $0\leq i<m$, be linear functionals
subject only to~$f^1(v)=1$. Then the linear map $R:\g\to\gl(m)$ defined by
$R(x)=D$, $R(u)=f^i(u)J^i$ for $u\in\a(i)$ and $0\leq i<m$, $R(u)=0$ for $u\in\a(\delta)$ and
$\delta\notin\{0,1,\dots,m-1\}$, yields an admissible uniserial
$\g$-module $U$ of length $m$ upon which $[\g,\g]$ does not act
trivially.

Moreover any modification whatsoever to $m,\al,f^0,\dots,f^{m-1}$
that keeps $f^1(v)=1$ produces a $\g$-module non-isomorphic to $U$.
\end{prop}

\begin{proof} The first assertion is obvious. As for the second,
$m$ is the composition length of $U$ and $\alpha$ is the eigenvalue through which $x$ acts
on the socle of $U$, so $m$ and $\al$ cannot be changed without changing the isomorphism type of $U$.
Since $R(v)=J$ and the only matrices in $\gl(m)$ commuting with $J$ are in $F[J]$,
the linear functionals $f^0,\dots,f^{m-1}$ cannot be changed either.
\end{proof}

\begin{theorem}\label{51}  Let
$U$ be an admissible uniserial $\g$-module of composition length
$m$ not annihilated by $[\g,\g]$. Suppose that $F$ has
characteristic 0 or prime characteristic $p\geq m$.
 Then

{\rm (1)} $m=\dim(U)>1$.

{\rm (2)} Let $S$ be the set of eigenvalues $\delta$ of $\ad\, x$
acting on $\a$ such that the corresponding eigenspace $\a(\delta)$
does not act trivially on $U$. Then, after replacing $x$ by a
suitable scalar multiple, $1\in S$ and every element of $S$ is
is an integer $i$ satisfying $0\leq i<m$.

{\rm (3)} For each $i\in S$ let $\w(i)$ be the subspace of $\a(i)$
annihilating $U$. Then $\w(i)$ is a hyperplane of $\a(i)$. In
particular, if $U$ is faithful then every eigenvalue of $\ad\, x$
acting on $\a$ must have multiplicity 1.

{\rm (4)} Fix a non-zero $v\in\a(1)$. Then there are linear
functionals $f^i:\a(i)\to F$, $0\leq i<m$, such that $f^1(v)=1$,
and a basis of $U$ relative to which $x$ is represented by a
diagonal matrix $D=\mathrm{diag}(\al,\al-1,\dots,\al-(m-1))$ for
some $\al\in F$ and every $u\in \a(i)$, $0\leq i<m$, is
represented by $f^i(u)J^i$, where $J=J_m(0)$.
\end{theorem}

\begin{proof} The centralizer of $J$ in $\gl(m)$ is $F[J]$ and
the eigenvalues of $\ad\, D$ on $F[J]$ are the integers $i$ such that $0\leq i<m$, with
corresponding eigenspaces $F\cdot J^i$. The result now follows from Proposition \ref{auxi}.
\end{proof}

\begin{note}{\rm Combining Proposition \ref{exist} with Theorem \ref{51} we obtain
a complete classification of all admissible uniserial $\g$-modules of composition length $m$ not
annihilated by $[\g,\g]$, provided $\chr(F)=0$ or $\chr(F)=p\geq m$.

Observe that part (2) of Theorem \ref{51} is slightly better than
the corresponding result from \cite{CS3}, Corollary 3.15, due to
our use of Theorem \ref{comm} instead of \cite{CS2}, Theorem 3.1.
}
\end{note}

\begin{prop}\label{exist2} Suppose that $F$ has prime characteristic $p$.
Assume that $1$ is an eigenvalue when $\ad\, x$ acts
on $\a$ and let $v\in\a$ be a corresponding eiegenvalue. Given
$m>p$ and $\al\in F$, set $D=\mathrm{diag}(\al,\al-1,\dots,\al-(m-1))$, $J=J_m(0)$ and let $Y\in D+F[J^p]J$.
Let $g^i:\a(i)\to F[J^p]J^i$, $0\leq i<p$, be arbitrary linear maps
subject only to $g^1(v)=I$.

Then the linear map $R:\g\to\gl(m)$ defined by $R(x)=Y$, $R(u)=g^i(u)$ for $u\in\a(i)$ and $0\leq i<p$,
$R(u)=0$ for $u\in\a(\delta)$ and
$\delta\notin\{0,\dots,p-1\}$, defines an admissible uniserial
$\g$-module $U$ of length $m$ upon which $[\g,\g]$ does not act
trivially.

Moreover, any modification whatsoever to $m,\al,Y,g^0,\dots,g^{m-1}$
that maintains $g^1(v)=I$ produces a $\g$-module non-isomorphic to $U$.
\end{prop}

\begin{proof} The fact that $R$ is a representation follows easily from (\ref{relpol}). In regards to the second assertion,
$m,\alpha,g^0,\dots,g^{p-1}$ cannot be changed without changing the isomorphism type of $U$ for the same reasons exposed
in the proof of Proposition~\ref{exist}. Moreover, by Proposition \ref{repe}, the elements of $D+F[J^p]J$
are in different orbits under the conjugating action of the unit group of $F[J]$, so $Y$ cannot be changed either.
\end{proof}

\begin{note}{\rm By varying $Y\in D+F[J^p]J$ while taking all $g^i(u)\in F\cdot J^i$
we obtain infinitely many non-isomorphic $\g$-modules, an
impossible feature when $p\geq m$. The only choice in $D+F[J^p]J$
that makes $x$ act diagonalizably on $U$ is $Y=D$ (it is
instructive to verify that none of the other matrices in
$D+F[J^p]J$ is diagonalizable). In stark contrast, when $p\geq m$
it is impossible for $x$ not to act diagonalizably on a uniserial
$\g$-module not annihilated by $[\g,\g]$, as shown in
Theorem \ref{51}.

Note, as well, that there are infinitely many choices of
$g^0,\dots,g^{p-1}$ such that $g^i(u)\notin F\cdot J^i$ for some
$i$ and $u$, all of which yield $\g$-modules which have no
analogue when $p\geq m$. }
\end{note}

\begin{theorem}\label{61} Suppose that $F$ has prime characteristic
$p$. Let $U$ is an admissible uniserial $\g$-module of composition
length $m>p$ not annihilated by $[\g,\g]$. Then

{\rm (1)} $m=\dim(U)>1$.

{\rm (2)} Let $S$ be the set of eigenvalues $\delta$ of $\ad\, x$
acting on $\a$ such that the corresponding eigenspace $\a(\delta)$
does not act trivially on $U$. Then, after replacing $x$ by a
suitable scalar multiple, $1\in S$ and every element of $S$ is
in the prime subfield of $F$.

{\rm (3)} Suppose that $U$ is faithful and let $i$ be an
eigenvalue of $\ad\, x$ acting on $\a$. Then
the multiplicity of $i$ is bounded above by
$[\frac{m-(i+1)}{p}]+1$.

{\rm (4)} Fix a non-zero $v\in\a(1)$. Then there are linear
maps $g^i:\a(i)\to F[J^p]J^i$, for $0\leq i<p$,
satisfying $g^1(v)=I$, a scalar $\al\in F$, a matrix $Y\in D+F[J^p]J$,
where $J=J_m(0)$ and $D=\mathrm{diag}(\al,\al-1,\cdots,\al-(m-1))$,
and a basis of $U$ relative to which $x$ is represented by $Y$ and every $u\in \a(i)$, $0\leq i<p$,
is represented by $g^i(u)$.
\end{theorem}

\begin{proof} The centralizer of $J$ in $\gl(m)$ is $F[J]$,
and the eigenvalues of $\ad\, D$ (and hence of $\ad\,Y$ for any $Y\in D+F[J^P]J$) on $F[J]$
comprise all of the prime subfield of $F$, with
corresponding eigenspaces $F[J^p]J^i$ for $0\leq i<p$. The result now follows from Proposition \ref{auxi}.
\end{proof}

\begin{note}{\rm Combining Proposition \ref{exist2} with Theorem \ref{61} we obtain
a complete classification of all admissible uniserial $\g$-modules
of composition length $m$ not annihilated by $[\g,\g]$ whenever
$F$ has prime characteristic $p< m$.}
\end{note}

%==================================================

\end{document}